\documentclass[11pt]{amsart}
\usepackage{latexsym, amssymb, stmaryrd, eucal}
\usepackage[T1]{fontenc}

\DeclareTextSymbol{\thh}{T1}{254}

\newtheorem{thm}{Theorem}[section]
\newtheorem{lemma}[thm]{Lemma}
\newtheorem{prop}[thm]{Proposition}
\newtheorem{cor}[thm]{Corollary}

\newtheorem{fact}[thm]{Fact}
\newtheorem{rmk}[thm]{Remark}

\theoremstyle{definition}
\newtheorem{df}[thm]{Definition}
\newtheorem{ex}[thm]{Example}

\newtheorem{question}[thm]{Question}

\newcommand{\BB}{\mathbb{B}}  
\newcommand{\BK}{\mathbb{K}}  
\newcommand{\BN}{\mathbb{N}}

\newcommand{\cu}[1]{\mathcal{#1}}

\newcommand{\bo}[1]{\boldsymbol{#1}}

\newcommand{\ti}[1]{\widetilde{#1}}

\newcommand{\sa}[1]{\mathsf{#1}}

\renewcommand{\hat}{\widehat}

\makeatletter

\def\indsym#1#2{%
  \setbox0=\hbox{$\m@th#1x$}%
  \kern\wd0%
  \hbox to 0pt{\hss$\m@th#1\mid$\hbox to 0pt{$\m@th#1^{#2}$}\hss}%
  \lower.9\ht0\hbox to 0pt{\hss$\m@th#1\smile$\hss}%
  \kern\wd0}
\newcommand{\ind}[1][]{\mathop{\mathpalette\indsym{#1}}}

\def\nindsym#1#2{%
  \setbox0=\hbox{$\m@th#1x$}%
  \kern\wd0%
  \hbox to 0pt{\hss$\m@th#1\not$\kern1.4\wd0\hss}
  \hbox to 0pt{\hss$\m@th#1\mid$\hbox to 0pt{$\m@th#1^{\,#2}$}\hss}%
  \lower.9\ht0\hbox to 0pt{\hss$\m@th#1\smile$\hss}%
  \kern\wd0}
\newcommand{\nind}[1][]{\mathop{\mathpalette\nindsym{#1}}}

\def\dotminussym#1#2{%
  \setbox0=\hbox{$\m@th#1-$}%
  \kern.5\wd0%
  \hbox to 0pt{\hss\hbox{$\m@th#1-$}\hss}%
  \raise.6\ht0\hbox to 0pt{\hss$\m@th#1.$\hss}%
  \kern.5\wd0}

\def \<{\langle}
\def \>{\rangle}
\def \((  {(\!(}
\def \)) {)\!)}

\def \cl {\operatorname{cl}}

\def \tp{\operatorname{tp}}

\def \acl{\operatorname{acl}}
\def \dcl{\operatorname{dcl}}

\def \DLO{\operatorname{DLO}}

\def \fo{\operatorname{fdcl}}

\numberwithin{equation}{section}
\def \l{\llbracket}
\def \rr{\rrbracket}

\allowdisplaybreaks[2]

\begin{document}

\title{Algebraic Independence Relations in Randomizations}

\author{Uri Andrews, Isaac Goldbring, and H. Jerome Keisler}

\address{University of Wisconsin-Madison, Department of Mathematics, Madison,  WI 53706-1388}
\email{andrews@math.wisc.edu}
\urladdr{www.math.wisc.edu/~andrews}
\email{keisler@math.wisc.edu}
\urladdr{www.math.wisc.edu/~keisler}
\address{University of California, Irvine, Department of Mathematics, Irvine, CA, 92697-3875}
\email{isaac@math.uci.edu}
\urladdr{www.math.uci.edu/~isaac}

\date{\today}

\begin{abstract}
We study the properties of algebraic independence and pointwise algebraic independence in a class of continuous theories, the randomizations $T^R$ of
complete first order theories $T$.
If algebraic and definable closure coincide in $T$, then algebraic independence in $T^R$ satisfies extension and has local character with the smallest
possible bound, but has neither finite character nor base monotonicity.  For arbitrary $T$, pointwise algebraic independence in $T^R$ satisfies extension
for countable sets, has finite character, has local character with the smallest possible bound, and satisfies base monotonicity if and only if
algebraic independence in $T$ does.
\end{abstract}

\maketitle

\section{Introduction}
The randomization of a complete first order theory $T$ is the complete continuous theory $T^R$
with two sorts, a sort for random elements of models of $T$, and a sort for events in an underlying probability space.
The aim of this paper is to investigate algebraic independence relations in randomizations of first order theories.
We will use results from our earlier papers [AGK1], which characterizes definability in randomizations, and [AGK2], where it is shown that the randomization of every
o-minimal theory is real rosy, that is, has a strict independence relation.

We focus on the independence axioms introduced by Adler [Ad2] (see Definition \ref{d-adler} below).
In first order model theory, algebraic independence is anti-reflexive and satisfies all of Adler's axioms except perhaps base monotonicity,
and also satisfies \emph{small local character}, a property that implies local character with the smallest possible bound $\kappa(D)=(|D|+\aleph_0)^+$.
It was shown in [BBHU] and [EG] that for any complete continuous theory, the algebraic independence relation satisfies all of the Adler's axioms except perhaps base
monotonicity, extension, and finite character, and also satisfies countable character (a weakening of finite character),  has local character with
bound $\kappa(D)=((|D|+2)^{\aleph_0})^+$, and is anti-reflexive.  We show here that if the underlying first order theory
$T$ has $\acl=\dcl$ (that is, algebraic closure coincides with definable closure), then algebraic closure in $T^R$ also satisfies extension and small local character.
However, for every $T$, algebraic independence in $T^R$ never has finite character and never satisfies base monotonicity.

Another relation on models of $T^R$ is \emph{pointwise algebraic independence}, which was introduced in [AGK2] and roughly means algebraic
independence almost everywhere.  We show that for arbitrary $T$ (rather than just when $T$ has $\acl=\dcl$), pointwise algebraic independence
in $T^R$ satisfies all of Adler's axioms except perhaps base monotonicity and extension.  In particular,
it does have finite character.  Moreover, pointwise algebraic independence satisfies extension for countable sets, has small local character, and satisfies base monotonicity
if and only if algebraic extension in $T$ satisfies base monotonicity.  However, pointwise algebraic independence is never anti-reflexive.

This paper is organized as follows.  In Section 2 we review Adler's axioms for independence relations and some general results from the literature about algebraic
independence in first order and continuous model theory.  Section 3 contains some notions and results about the randomization theory $T^R$ that we will need from the papers
[AGK1] and [AGK2].  Section 4 contains the proofs of the negative results that in $T^R$, algebraic independence never has finite character and never satisfies
base monotonicity.  To better understand why this happens, we take a closer look at the example of dense linear order.  Section 5 contains the proof of the result
that if $T$  has $\acl=\dcl$ then algebraic independence in $T^R$ satisfies the extension axiom.   In Section 6 we prove that if $T$ has $\acl=\dcl$ then
algebraic independence in $T^R$ has small local character.  On the way to this proof, we introduce the pointwise algebraic independence relation in $T^R$,
and show that it has small local character whether or not $T$ has $\acl=\dcl$.  Finally, in Section 7 we prove the other results stated in the
preceding  paragraph about pointwise algebraic independence in $T^R$. We also show that in $T^R$, pointwise algebraic independence never implies algebraic independence,
and algebraic independence implies pointwise algebraic independence only in the trivial case that the models of $T$ are finite.

For background in continuous model theory in its current form we refer to the papers [BBHU] and [BU].
We assume the reader is familiar with the basics of continuous model theory, including the notions of a theory, model, pre-model, reduction, and completion.
For background on randomizations of models we refer to the papers [Ke] and [BK].
We follow the terminology of [AGK2]. A continuous pre-model is called \emph{pre-complete} if its reduction is its completion.  The set of all finite tuples in a set $A$
is denoted by $A^{<\BN}.$  We assume throughout this paper that $T$ is a complete first order theory with countable signature $L$ and models of cardinality $>1$, and
that $\upsilon$ is an uncountable inaccessible cardinal that is held fixed.  We let $\cu M$ be the \emph{big model} of $T$, that is,
the (unique up to isomorphism) saturated model $\cu M\models T$ that is finite or of cardinality $|\cu N|=\upsilon$.
We call a set \emph{small} if it has cardinality $< \upsilon$, and \emph{large} otherwise.

\section{Independence}

\subsection{Abstract Independence Relations}

Since the various properties of independence are given some slightly different names in
various parts of the literature, we take this opportunity to declare that we are following the terminology established in [Ad2], which is
repeated here for the reader's convenience.  In this paper, we will sometimes write $AB$ for $A\cup B$, and write $[A,B]$ for
$\{D\colon A\subseteq D \wedge D\subseteq B\}$

\begin{df}[Adler]  \label{d-adler}
Let $\cu N$ be the big model of a continuous or first order theory.  By a \emph{ternary relation over $\cu N$} we mean
a ternary relation $\ind$ on the small subsets of $\cu N$.  We say that $\ind$
is an \emph{independence relation} if it satisfies the following \emph{axioms for independence relations} for all small sets:
\begin{enumerate}
\item (Invariance) If $A\ind_CB$ and $(A',B',C')\equiv (A,B,C)$, then $A'\ind_{C'}B'$.
\item (Monotonicity) If $A\ind_C B$, $A'\subseteq A$, and $B'\subseteq B$, then $A'\ind_C B'$.
\item (Base monotonicity) Suppose $C\in[D, B]$.  If $A\ind_D B$, then $A\ind_C B$.
\item (Transitivity) Suppose $C\in[D, B]$.  If $B\ind_C A$ and $C\ind_D A$, then $B\ind_D A$.
\item (Normality) $A\ind_CB$ implies $AC\ind_C B$.
\item (Extension) If $A\ind_C B$ and $\hat B\supseteq B$, then there is $A'\equiv_{BC} A$ such that $A'\ind_C \hat B$.
\item (Finite character) If $A_0\ind_CB$ for all finite $A_0\subseteq A$, then $A\ind_CB$.
\item (Local character) For every $A$, there is a cardinal $\kappa(A)<\upsilon$ such that, for any set $B$, there is a subset $C$ of $B$ with $|C|<\kappa(A)$ such that $A\ind_CB$.
\end{enumerate}
If finite character is replaced by countable character (which is defined in the obvious way), then we say that $\ind$ is a
\emph{countable independence relation}.  We will refer to the first five axioms (1)--(5) as the \emph{basic axioms}.
\end{df}

\begin{df}
An independence relation $\ind$ is \emph{strict} if it satisfies
\begin{itemize}
\item[(9)] (Anti-reflexivity) $a\ind_B a$ implies $a\in \acl(B)$.
\end{itemize}
\end{df}

There are two other useful properties to consider when studying ternary relations over $\cu N$:

\begin{df}

\noindent\begin{enumerate}
\item[(10)] (Full existence) For every $A,B,C$, there is $A'\equiv_C A$ such that $A'\ind_C B$.
\item[(11)] (Symmetry) For every $A,B,C$, $A\ind_CB$ implies $B\ind_CA$.
\end{enumerate}
\end{df}

\begin{fact}  \label{f-fe-vs-ext}  (Remarks 2.2.4 in [AGK2]).

\noindent\begin{enumerate}
\item[(i)]  Whenever $\ind$ satisfies invariance, monotonicity, transitivity, normality, full existence, and symmetry,
then $\ind$ also satisfies extension.
\item[(ii)]  Any countable independence relation is symmetric.
\end{enumerate}
\end{fact}

\begin{df}

\noindent\begin{enumerate}
\item[(i)] We say that $\ind$ has \emph{countably local character} if for every countable set $A$ and every small set $B$, there is a countable subset $C$ of $B$ such that $A\ind_CB$.
\item[(ii)] We say that $\ind$ has \emph{small local character} if  for all small sets $A, B, C_0$
such that $C_0\subseteq B$ and $|C_0|\le|A|+\aleph_0$, there is
a set $C\in[C_0,B]$ such that $|C|\le|A|+\aleph_0$ and $A\ind_C \ B$.
\end{enumerate}
\end{df}

\begin{fact}  \label{f-countably-localchar} (Remark 2.2.7 in [AGK2]).
\noindent\begin{enumerate}
\item[(i)]  If $\ind$ has small local character, then $\ind$ has local character with bound $\kappa(D)=(|D|+\aleph_0)^+$ (the smallest possible bound).
In the presence of monotonicity, the converse is also true.
\item[(ii)] If $\ind$ has local character with bound $\kappa(D)=(|D|+\aleph_0)^+$, then $\ind$ has countably local character.
\item[(iii)] If $\ind$ has invariance, countable character, base monotonicity, and countably local character, then
$\ind$ has local character with bound $\kappa(D)=((|D| + 2)^{\aleph_0})^+$.
\end{enumerate}
\end{fact}

We say that $\ind[J]$ is \emph{weaker than} $\ind[I]$, and write $\ind[I]\Rightarrow \ind[J]$, if $A\ind[I]_C B\Rightarrow A\ind[J]_C B$.

\begin{rmk}  \label{r-weaker}  Suppose $\ind[I]\Rightarrow\ind[J]$.  If $\ind[I]$ has full existence, local character,
countably local character, or small local character.  Then $\ind[J]$ \ has the same property.
\end{rmk}

\subsection{Algebraic Independence}

\begin{df}   In first order logic, a formula $\varphi(u,\vec v)$ is \emph{functional} in  $T$ if
$$T\models(\forall \vec v)(\exists ^{\le  1} u)\varphi(u,\vec v).$$
$\varphi(u,\vec v)$ is \emph{algebraical} in $T$ if there exists $n\in\BN$ such that
$$T\models(\forall \vec v)(\exists ^{\le  n} u)\varphi(u,\vec v).$$
\end{df}

The \emph{definable closure} of $A$ in  $\cu M$  is the set
$$\dcl^{\cu M}(A)=\{b\in M\mid \cu M\models\varphi(b,\vec a) \mbox{ for some functional } \varphi \mbox{ and } \vec a\in A^{<\BN}\}.$$
The \emph{algebraic closure} of $A$ in $\cu M$ is the set
$$\acl^{\cu M}(A)=\{b\in M\mid \cu M\models\varphi(b,\vec a) \mbox{ for some algebraical } \varphi \mbox{ and } \vec a\in A^{<\BN}\}.$$

We refer to [BBHU] for the definitions of the algebraic closure $\acl^{\cu N}(A)$ and definable closure $\dcl^{\cu N}(A)$ in a continuous structure $\cu N$.
If  $\cu N$ is clear from the context, we will sometimes drop the superscript and write $\dcl, \acl$ instead of $ \dcl^\cu N, \acl^\cu N$.
We will often use the following facts without explicit mention.

\begin{fact}  \label{f-algebraic-cardinality}  (Follows from [BBHU], Exercise 10.8)
For every set $A$, $\acl(A)$ has cardinality at most $( |A|+2)^{\aleph_0}$.  Thus the algebraic closure of a small set is small.
\end{fact}

\begin{fact} \label{f-definableclosure}  (Definable Closure, Exercises 10.10 and 10.11, and Corollary 10.5 in [BBHU])
\begin{enumerate}
\item If $ A\subseteq\cu N$ then $\dcl( A)=\dcl(\dcl( A))$ and $\acl( A)=\acl(\acl( A))$.
\item If $ A$ is a dense subset of the topological closure of $B$ and $ B\subseteq\cu N$, then $\dcl(A)=\dcl( B)$ and $\acl(A)=\acl( B)$.
\end{enumerate}
\end{fact}

It follows that for any $ A\subseteq\cu N$, $\dcl( A)$ and $\acl( A)$ are topologically closed.

\noindent In any complete theory (first order or continuous), we define the notion of \emph{algebraic independence}, denoted $\ind[a]$,
by setting $A\ind[a]_CB$ to mean $\acl(AC)\cap \acl(BC)=\acl(C)$. In first order logic, $\ind[a]$ satisfies all axioms for a strict
independence relation except for perhaps base monotonicity.

\begin{prop}  \label{p-alg-indep}
In any complete continuous theory, $\ind[a]$ satisfies symmetry and all axioms for a strict countable independence relation except perhaps for
base monotonicity and extension.
\end{prop}

\begin{proof}
The proof is exactly as in [Ad2], Proposition 1.5, except for some minor modifications.  For example, countable character of $\acl$ in
continuous logic yields countable character of $\ind[a]$.  Also, in the verification of local character, one needs to take
$\kappa(A):=((|A|+2)^{\aleph_0})^+$ instead of $(|A|+\aleph_0)^+$.
\end{proof}

\begin{question}  \label{q-a-full-existence}
Does $\ind[a]$ always have full existence (or extension) in continuous logic?
\end{question}

The proof that $\ind[a]$  has full existence in first order logic uses the negation connective, which is not available in continuous logic.

\section{Randomizations}

\subsection{The Theory $T^R$}

The \emph{randomization signature} $L^R$ is the two-sorted continuous signature
with sorts $\BK$ (for random elements) and $\BB$ (for events), an $n$-ary
function symbol $\l\varphi(\cdot)\rr$ of sort $\BK^n\to\BB$
for each first order formula $\varphi$ of $L$ with $n$ free variables,
a $[0,1]$-valued unary predicate symbol $\mu$ of sort $\BB$ for probability, and
the Boolean operations $\top,\bot,\sqcap, \sqcup,\neg$ of sort $\BB$.  The signature
$L^R$ also has distance predicates $d_\BB$ of sort $\BB$ and $d_\BK$ of sort $\BK$.
In $L^R$, we use ${\sa B},{\sa C},\ldots$ for variables or parameters of sort $\BB$. ${\sa B}\doteq{\sa C}$
means $d_\BB({\sa B},{\sa C})=0$, and ${\sa B}\sqsubseteq{\sa C}$ means ${\sa B}\doteq{\sa B}\sqcap{\sa C}$.
A structure with signature $L^R$ will be a pair $\cu N=(\cu K,\cu E)$ where $\cu K$ is the part of sort $\BK$ and
$\cu E$ is the part of sort $\BB$.

The following fact, which is a consequence of Proposition 2.1.10 of [AGK1], gives a model-theoretic characterization of $T^R$.

\begin{fact}  \label{f-neat}
There is a unique complete theory $T^R$ with signature $L^R$ whose big model $\cu N=(\cu K,\cu E)$ is the reduction of a
pre-complete-structure $\cu P=(\cu J,\cu F)$
equipped with a complete atomless probability space $(\Omega,\cu F,\mu )$ such that:
\begin{enumerate}
\item $\cu F$ is a $\sigma$-algebra with $\top,\bot,\sqcap, \sqcup,\neg$ interpreted by $\Omega,\emptyset,\cap,\cup,\setminus$.
\item $\cu J$ is a set of functions $a\colon\Omega\to M$.
\item For each formula $\psi(\vec{x})$ of $L$ and tuple
$\vec{a}$ in $\cu J$, we have
$$\l\psi(\vec{a})\rr=\{\omega\in\Omega:\cu M\models\psi(\vec{a}(\omega))\}\in\cu F.$$
\item $\cu F$ is equal to the set of all events
$ \l\psi(\vec{a})\rr$
where $\psi(\vec{v})$ is a formula of $L$ and $\vec{a}$ is a tuple in $\cu J$.
\item  For each formula $\theta(u, \vec{v})$
of $L$ and tuple $\vec{b}$ in $\cu J$, there exists $a\in\cu J$ such that
$$ \l \theta(a,\vec{b})\rr=\l(\exists u\,\theta)(\vec{b})\rr.$$
\item On $\cu J$, the distance predicate $d_\BK$ defines the pseudo-metric
$$d_\BK(a,b)= \mu \l a\neq b\rr .$$
\item On $\cu F$, the distance predicate $d_\BB$ defines the pseudo-metric
$$d_\BB({\sa B},{\sa C})=\mu ( {\sa B}\triangle {\sa C}).$$
\end{enumerate}
\end{fact}

\begin{fact}  \label{f-glue}  (Lemma 2.1.8 in [AGK1])
In the  big model $\cu N=(\cu K,\cu E)$ of $T^R$, for each  $\bo a ,\bo b \in\cu K$ and ${\sa B}\in\cu E$,
there is an element $\bo c \in\cu K$ that agrees with $\bo a $ on ${\sa B}$ and agrees with $\bo b $ on $\neg{\sa B}$,
that is, ${\sa B}\sqsubseteq\l \bo c =\bo a \rr$ and $(\neg{\sa B})\sqsubseteq\l \bo c =\bo b \rr$.
\end{fact}

\begin{df}  A first order or continuous theory  has $\acl=\dcl$ if $\acl(A)=\dcl(A)$ for every set $A$ in every model of the theory.
\end{df}

For example, any first order theory $T$ with a definable linear ordering has $\acl=\dcl$.

\begin{fact}  \label{f-acl=dcl}  ([AGK1], Proposition 3.3.7, see also [Be2])

In the big model $\cu N$ of $T^R$, $\acl_\BB( A)=\dcl_\BB( A)$ and $\acl( A)=\dcl( A)$.  Thus $T^R$ has $\acl=\dcl.$
\end{fact}

As a corollary, we obtain the following characterization of algebraic independence in $\cu N$.

\begin{cor}  \label{c-alg-indep} In the big model $\cu N$ of $T^R$,  $ A\ind[a]_C B$ if and only if
$$[\dcl(AC)\cap\dcl(BC)=\dcl(C)] \wedge [\dcl_\BB(AC)\cap\dcl_\BB(BC)=\dcl_\BB(C)].$$
\end{cor}

\begin{proof}  By the definition of algebraic independence in the two-sorted metric structure $\cu N$ and Fact \ref{f-acl=dcl}.
\end{proof}

From now on we will work within the big model $\cu N=( {\cu K}, {\cu E})$ of $T^R$.
By saturation, $\cu K$ and $\cu E$ are large.  Hereafter, $A, B, C$ will always denote small subsets of $\cu K$, and $\cu N_A$ will denote
the expansion of $\cu N$ formed by adding a constant symbol for each $\bo a\in A$.
We will write $\dcl, \acl$ for $\dcl^{\cu N}, \acl^{\cu N}$, and $\ind[a]$ will denote the algebraic independence relation in $\cu N$.

For each element $\bo b\in {\cu K}$, we will also choose once and for all a \emph{representative} $b\in\cu J$ such that
the image of $b$ under the reduction map is $\bo b$.  It follows that for each first order formula $\varphi(\vec v)$,
$\l\varphi(\vec{\bo a})\rr$ in $\cu N$ is the image of $\l\varphi(\vec a)\rr$ in $\cu P$ under the reduction map.
Note that any two representatives of an element $\bo b\in\cu K$ agree except on a set of measure zero.

For any small $A\subseteq\cu K$ and each $\omega\in\Omega$, we define
$$ A(\omega)=\{a(\omega)\mid \bo a\in  A\},$$
and let $\cl( A)$ denote the closure of $ A$ in the metric $d_\BK$.  When $\cu A\subseteq {\cu E}$,
$\cl(\cu A)$ denotes the closure of $\cu A$ in the metric $d_\BB$, and
$\sigma(\cu A)$ denotes the smallest $\sigma$-subalgebra of $ {\cu E}$ containing $\cu A$.
Since the cardinality $\upsilon$ of $\cu N$ is inaccessible, whenever $A\subseteq\cu K$ is small, the closure $\cl(A)$ and
the set of $n$-types over $A$ is small.  Also, whenever $\cu A\subseteq\cu E$ is small, the closure $\cl(\cu A)$ is small.

\subsection{Definability in $T^R$}

In this section we review some notions and results about definability that we will need from the paper [AGK1].
We write $\dcl_\BB( A)$ for the set of elements of sort $\BB$ that are definable over $ A$ in $\cu N$,
and write $\dcl( A)$ for the set of elements of sort $\BK$ that are definable over $ A$ in $\cu N$.
Similarly for $\acl_\BB( A)$ and $\acl( A)$.

\begin{df}  We say that an event $\sa E$ is \emph{first order definable over $A$}, in symbols $\sa E\in\fo_\BB(A)$, if
$\sa E=\l\theta(\vec{\bo a})\rr$ for some formula $\theta$ of $L$ and some tuple $\vec{\bo a}\in A^{<\BN}$.
\end{df}

\begin{df}
We say that $\bo b$ is \emph{first order definable over $ A$}, in symbols $\bo b\in\fo( A)$, if there is a functional formula
$\varphi(u,\vec v)$ and a tuple $\vec{\bo a}\in {A}^{<\BN}$ such that
$\l \varphi(b,\vec{a})\rr=\top$.
\end{df}

\begin{fact} \label{f-separable}  ([AGK1], Theorems 3.1.2 and  3.3.6)
$$\dcl_\BB( A)=\cl(\fo_\BB( A))=\sigma(\fo_\BB(A))\subseteq\cu E,\qquad \dcl( A)=\cl(\fo( A))\subseteq\cu K.$$
\end{fact}

It follows that whenever $A$ is small, $\dcl(A)$ and $\dcl_\BB(A)$ are small.

\begin{rmk}  \label{r-dcl-B}  For each small $A$,
$$\fo_\BB(\fo(A))=\fo_\BB(A),\quad\dcl_\BB(\dcl(A))=\dcl_\BB(A).$$
\end{rmk}

We will sometimes use the $\l \ldots\rr$ notation in a general setting.  Given a property $P(\omega)$, we write
$$\l P\rr=\{\omega\in\Omega\,:\,P(\omega)\}.$$

\begin{df} \label{d-pointwise-def}
 We say that $\bo b$ is \emph{pointwise definable over $A$}, in symbols $\bo b\in\dcl^\omega(A)$, if
$$\mu(\l\bo b\in\dcl^{\cu M}(A_0)\rr)=1$$
for some countable $A_0\subseteq A$.

We say that $\bo b$ is \emph{pointwise algebraic over $A$}, in symbols $\bo b\in\acl^\omega(A)$, if
$$\mu(\l \bo b\in\acl^{\cu M}(A_0)\rr)=1$$
for some countable $A_0\subseteq A$.
\end{df}

\begin{rmk}  \label{r-dcl-omega}
$\dcl^\omega$ and $\acl^\omega$ have countable character, that is,
$\bo b\in\dcl^\omega(A)$ if and only if $\bo b\in\dcl^\omega(A_0)$ for some countable $A_0\subseteq A$, and similarly for $\acl^\omega$.
\end{rmk}

The next result is a useful characterization of $\dcl(A)$.

\begin{fact}  \label{f-dcl3}  ([AGK1], Corollary 3.3.5) For any element $\bo b\in {\cu K}$,
$\bo{b}$ is definable over $ A$ if and only if:
\begin{enumerate}
\item $\bo b$ is pointwise definable over $A$;
\item $\fo_\BB(\bo b A)\subseteq\dcl_\BB(A).$
\end{enumerate}
\end{fact}

\begin{cor}  \label{c-pointwise-alg-def}
In $\cu N$ we always have
$$\acl(A)=\dcl(A)\subseteq\dcl^\omega(A)=\dcl^\omega(\dcl^\omega(A))\subseteq\acl^\omega(A)=\acl^\omega(\acl^\omega(A)).$$
\end{cor}

\subsection{Algebraic Independence in the Event Sort}

The ternary relation $\ind[a\BB] \ \ $ on the big model $\cu N$ of $T^R$  was introduced in the paper [AGK2] and will be useful here.  It is the
analogue of algebraic independence obtained by restricting the algebraic closures of sets to the event sort.

\begin{df}  \label{d-event-indep}
For small $A, B, C\subseteq \cu K$, define
 $$A\ind[a\BB]_C \ \ B\Leftrightarrow\acl_\BB(AC)\cap\acl_\BB(BC)=\acl_\BB(C).$$
\end{df}

\begin{rmk}  \label{r-acl=dcl}
By Fact \ref{f-acl=dcl}, for small $A, B, C\subseteq \cu K$, we have
 $$A\ind[a\BB]_C \ \ B\Leftrightarrow\dcl_\BB(AC)\cap\dcl_\BB(BC)=\dcl_\BB(C).$$
By Corollary \ref{c-alg-indep}, we also have
$$A\ind[a]_C B \Leftrightarrow (\dcl(AC)\cap\dcl(BC)=\dcl(C))\wedge A\ind[a\BB]_C \ B.$$
\end{rmk}

\begin{fact} \label{f-event-aB}  (Proposition 6.2.4 in [AGK2]).
In  $T^R$, the relation $\ind[a\BB] \ \ $ satisfies all the axioms for a countable independence relation except base monotonicity.
It also has symmetry and small local character.
\end{fact}

We will also need the following fact, which is given by Lemma 6.1.6, Corollary 6.1.7, and Lemma 6.2.3 of [AGK2], and is a consequence of a result in [Be].

\begin{fact}  \label{f-event-dB}
There is a countable independence relation $\ind[d\BB] \ $ (dividing independence in the event sort) that has small local character over $\cu N$
and is such that $\ind[d\BB] \ \Rightarrow \ind[a\BB] \ .$
\end{fact}

\section{Negative Results: Finite Character and Base Monotonicity}

In this section we show that for \emph{every} $T$,  algebraic independence in $T^R$ satisfies neither finite character nor base monotonicity.
The following lemmas and notation  will be useful for these results.

\begin{lemma}  \label{l-char-function}
There exists a pair $Z=\{\bo{0},\bo{1}\}\subseteq\cu K$ such that $\l\bo 0\ne \bo 1\rr=\top$, and $\dcl_\BB(Z)=\{\top,\bot\}.$
\end{lemma}

\begin{proof}
This follows from Fact \ref{f-neat} and the fact that $\cu N$ is saturated.
\end{proof}

By Fact \ref{f-glue}, for each event $\sa B\in\cu E$, there is a unique element
$1_{\sa B}\in\cu K$ that agrees with $1$ on $\sa B$ and agrees with $0$ on $\neg\sa B$.
Given a set $\cu A\subseteq\cu E$, let $\bo 1_{\cu A}=\{\bo 1_{\sa B} \mid \sa B\in\cu A\}.$

\begin{lemma} \label{l-char-functions}
If  $\cu A\subseteq\cu E$, then $\dcl_\BB(\bo 1_{\cu A} Z)=\sigma(\cu A).$
\end{lemma}

\begin{proof}  By definition, $\sa E\in\fo_\BB(\bo 1_{\cu A}Z)$ if and only if $\sa E=\l\theta(\vec{\bo a},\bo 0,\bo 1)\rr$
for some formula $\theta\in L$ and tuple $\vec{\bo a}\subseteq\bo 1_{\cu A}.$  For each $\bo b\in\bo 1_{\cu A}$, we have $\mu(\l\bo b\in Z\rr)=1.$
It follows that $\l\theta(\vec{\bo a},\bo 0,\bo 1)\rr\in\sigma(\cu A)$.  By Fact \ref{f-separable}, $\dcl_\BB(\bo 1_{\cu A} Z)\subseteq\sigma(\cu A).$
For each $\sa B\in\cu A$ we have $\sa B=\l \bo 1_{\sa B} = \bo 1\rr,$ so $\cu A\subseteq \fo_\BB(\bo 1_{\cu A} Z).$
Then by Fact \ref{f-separable} again, $ \sigma(\cu A)\supseteq\dcl_\BB(\bo 1_{\cu A} Z).$
\end{proof}

\subsection{Finite Character}

\begin{prop}  \label{p-no-finite-character}
For every $T$, $\ind[a]$ in $T^R$ does not satisfy finite character.
\end{prop}

\begin{proof}  Since $\mu$ is atomless, there is an event $\sa B$ and a sequence of events $\<\sa B_n\>_{n\in\BN}$ such that for each $n$,
$$  \sa B_n\sqsubseteq\sa B_{n+1}\sqsubseteq\sa B, \quad \mu(\sa B\setminus\sa B_n)=2^{-(n+1)},\quad \mu(\sa B)=1/2.$$
Let  $\bo b=\bo 1_{\sa B}$, $\cu A_n=\{\sa B_m\mid m<n\}$, $\cu A=\{\sa B_m\mid m\in\BN\}.$
Then $\bo 1_{\cu A}=\bigcup_n \bo 1_{\cu A_n}.$ By Lemma \ref{l-char-functions},
$$\dcl_{\BB}(\bo 1_{\cu A_n} Z)=\sigma(\cu A_n), \quad \dcl_{\BB}(\bo b Z)=\sigma(\{\sa B\})\subseteq\sigma(\bo 1_{\cu A} Z)=\dcl_{\BB}(\cu A).$$
Note that every element of $\cu K$ that is pointwise definable from $\bo 1_{\cu A} Z$ is pointwise definable from $Z$.
Then by Fact \ref{f-dcl3}, we have
$$ \dcl( Z)=\{\bo x\in\dcl^\omega(Z)\mid  \fo_\BB(\bo x  Z)\subseteq \{\top,\bot\}\},$$
$$ \dcl(\bo 1_{\cu A} Z)=\{\bo x\in\dcl^\omega(Z)\mid \fo_\BB(\bo x \bo 1_{\cu A} Z)\subseteq \sigma(\cu A)\},$$
$$ \dcl(\bo 1_{\cu A_n} Z)=\{\bo x\in\dcl^\omega(Z)\mid \fo_\BB(\bo x \bo 1_{\cu A_n} Z)\subseteq \sigma(\cu A_n)\},$$
$$ \dcl(\bo b Z)=\{\bo x\in\dcl^\omega(Z)\mid  \fo_\BB(\bo x \bo b Z)\subseteq \sigma(\sa B)\}.$$
But $\sigma(\cu A_n)\cap\sigma(\{\sa B\})=\{\top,\bot\}$, so by Fact \ref{f-acl=dcl},
$$\acl(\bo 1_{\cu A_n} Z)\cap\acl(\bo b Z)=\dcl(\bo 1_{\cu A_n} Z)\cap\dcl(\bo b Z)=\dcl(Z)=\acl(Z),$$
and hence $\bo 1_{\cu A_n}\ind[a]_Z \bo b.$

However,  $\sa B\in\sigma(\sa A)$, so by Lemma \ref{l-char-functions} we have
$\dcl_\BB(\bo b \bo 1_{\cu A} Z)=\sigma(\cu A).$  Moreover, $\bo b\in\dcl^\omega(Z)$.  Therefore
$$\bo b\in \acl(\bo 1_{\cu A} Z)\cap\acl(\bo b)\setminus \acl(Z),$$
so $\bo 1_{\cu A} \nind[a]_Z \bo b$ and finite character fails.
\end{proof}

\subsection{Base Monotonicity}

By Proposition 1.5 (3) in [Ad1], for any complete first order theory $T$, $\ind[a]$ satisfies base monotonicity if and only if the lattice of algebraically closed sets is modular.
The argument there shows that the same result holds for any complete continuous theory.
We show that for $T^R$, $\ind[a]$ never satisfies base monotonicity, and thus is never modular and is never a countable independence relation.

\begin{prop}  \label{p-a-nobase-monotonicity}  For every $T$,
$\ind[a]$ in $T^R$ does not satisfy base monotonicity.
\end{prop}

\begin{proof}  Since $\mu$ is atomless, there are two
independent events $\sa D, \sa F$ in $\cu E$ of probability $1/2$. Let $\sa E=\sa D\sqcap\sa F$. $\bo a=1_{\sa D}$, $\bo b = 1_{\sa E}$, and $\bo c=1_{\sa F}$.  Then
$$\dcl_\BB(\bo a)=\sigma(\{\sa D\}),\quad
 \dcl_\BB(\bo c)=\sigma(\{\sa F\}),$$
$$ \dcl_\BB(\bo a\bo c)=\sigma(\{\sa D,\sa F\}),\quad
 \dcl_\BB(\bo b \bo c)=\sigma(\{\sa E,\sa F\}).$$
As in the proof of Proposition \ref{p-no-finite-character}, we have
$$ \dcl( Z)=\{\bo x\in\dcl^\omega(Z)\mid  \fo_\BB(\bo x  Z)\subseteq \{\top,\bot\}\},$$
$$ \dcl(\bo a Z)=\{\bo x\in\dcl^\omega(Z)\mid \fo_\BB(\bo x\bo a Z)\subseteq\sigma(\{\sa D\})\},$$
$$ \dcl(\bo c Z)=\{\bo x\in\dcl^\omega(Z)\mid\\fo_\BB(\bo x\bo c Z)\subseteq\sigma(\{\sa F\})\},$$
$$ \dcl(\bo a \bo c Z)=\{\bo x\in\dcl^\omega(Z)\mid \fo_\BB(\bo x\bo a\bo c Z)\subseteq\sigma(\{\sa D,\sa F\})\},$$
$$ \dcl(\bo b \bo c Z)=\{\bo x\in\dcl^\omega(Z)\mid\fo_\BB(\bo x\bo b\bo c Z)\subseteq\sigma(\{\sa E,\sa F\})\}.$$
It follows that $\bo a\ind[a]_Z \bo b\bo c Z$. But
$$\sa E\in\sigma(\{\sa D,\sa F\})\cap\sigma(\{\sa E,\sa F\})\setminus\sigma(\{\sa F\}),$$
so
$$\bo b\in\acl(\bo a\bo c Z)\cap \acl(\bo b\bo c Z)\setminus \acl(\bo c Z).$$
Therefore  $\bo a\nind[a]_{\bo c Z}  \bo b\bo c Z$, and base monotonicity fails.
\end{proof}

Recall from [Ad1] that the ternary relation $\ind[M] \ $ is defined by
$$A\ind[M]_C \ B \ \Leftrightarrow (\forall D\in[C,\acl(BC)]) A\ind[a]_D B.$$
$\ind[M] \ $ is the weakest ternary relation that implies $\ind[a]$ and satisfies base monotonicity.
So in both first order and continuous model theory, if $\ind[a]$ satisfies base monotonicity then $\ind[M] \ =\ind[a]$
and hence $\ind[M] \ $ satisfies symmetry.  Thus the following corollary is an improvement of Proposition \ref{p-a-nobase-monotonicity}.

\begin{cor} \label{c-no-symmetry} For every $T$, $\ind[M] \ \ $ in $T^R$ does not satisfy symmetry.
\end{cor}

\begin{proof}
We use the notation introduced in the proof of  \ref{p-a-nobase-monotonicity}.  Since $\bo a\nind[a]_{\bo c Z}  \bo b\bo c Z$ and
$\ind[M]\Rightarrow\ind[a]$, we have $\bo a\nind[M]_{\bo c Z} \ \bo b\bo c Z.$  However, it follows from the proof of \ref{p-a-nobase-monotonicity} that
$\bo b\bo c Z\ind[M]_Z \ \bo a,$ so $\ind[M] \ $ does not satisfy symmetry.
\end{proof}

As an example, we look at the relations $\ind[a]$ and $\ind[M] \ $ in the continuous theory $\DLO^R$, the randomization of the theory of dense linear order without endpoints.
We will see that these relations are much more complicated in $\DLO^R$ than they are in $\DLO$.

\begin{ex}  Let $T=\DLO$.  In the big model $\cu M$ of $\DLO,$ we have $\acl(A)=\dcl(A)=A$ for every set $A$.  Thus in $\cu M$ the lattice
of algebraically closed sets is modular, $\ind[a]=\ind[M]  \ $, and $\ind[a]$ is a strict independence relation.

In the big model $\cu N$ of $\DLO^R$, $\ind[a]$ does not satisfy base monotonicity by Proposition \ref{p-a-nobase-monotonicity}, and $\ind[M] \ $ does not satisfy symmetry by
Corollary \ref{c-no-symmetry}.  Proposition 4.2.3 of [AGK1] shows that for every finite
set $A\subseteq\cu K$, $\dcl(A)$ is the smallest set $D\supseteq A$ such that whenever $\bo a, \bo b, \bo c, \bo d\in D$, the element of $\cu K$
that agrees with $\bo c$ on  $\l \bo a <\bo b\rr$ and agrees with  $\bo d$ on $\neg\l \bo a < \bo b\rr$ belongs to $D$.

Let $\bo a\vee\bo b$ and $\bo a\wedge\bo b$ denote the pointwise maximum and minimum, respectively.
We leave it to the reader to work out the following characterizations
of $A\ind[a]_C B$ and $A\ind[M]_C B$ in the simple case that $A, B, C$ are singletons in $\cu N$.
\begin{enumerate}
\item $\bo{a}\ind[a]_\emptyset \bo{b}\Leftrightarrow\bo{a}\not=\bo{b}.$
\item  $\acl(\bo a\bo b)=\{\bo a,\bo b,\bo a\vee\bo b,\bo a\wedge\bo b\}$.
\item $\bo{a}\ind[a]_{\bo{c}}\bo{b}\Leftrightarrow
\{\bo{a},\bo{c},\bo{a}\vee\bo{c},\bo{a}\wedge\bo{c}\}\cap \{\bo{b},\bo{c},\bo{b}\vee\bo{c},\bo{b}\wedge\bo{c}\}=\{\bo{c}\}.$
\end{enumerate}
To see where base monotonicity fails for $\ind[a]$, let $\sa E$ be an event with $0<\mu(\sa E)<1$ and take $\bo a, \bo b,\bo c$ so that
$\bo a=\bo b<\bo c$ on $\sa E$ and $\bo c<\bo a<\bo b$ on $\neg\sa E$. Then use (1) and (3) to show that $\bo a\ind[a]_{\emptyset} \bo b$ but
$\bo a\nind[a]_{\bo c} \bo b$.
\begin{itemize}
\item[(4)] If $\bo b\in\{\bo b\vee\bo c,\bo b\wedge\bo c\}$,  then $\bo{a}\ind[M]_{\bo{c}}\bo{b}\Leftrightarrow\bo{a}\ind[a]_{\bo{c}}\bo{b}$.
\item[(5)]  If $\bo b\notin\{\bo b\vee\bo c,\bo b\wedge\bo c\}$, then $\bo{a}\ind[M]_{\bo{c}}\bo{b}$ if and only if:
$$\bo{a}\ind[a]_{\bo{c}}\bo{b}, \quad \bo{b}\notin \dcl(\{\bo{a},\bo{c},\bo{b}\wedge\bo c\}), \quad \bo{b}\notin \dcl(\{\bo{a},\bo{c},\bo{b}\vee\bo c\}).$$
\end{itemize}
To see where  symmetry fails for $\ind[M] \ $, partition $\Omega$ into three events $\{\sa D, \sa E, \sa F\}$  of positive
measure.  Take $\bo a, \bo b, \bo c$ so that $\bo a=\bo b<\bo c$ in $\sa D$, $\bo a<\bo c<\bo b$ in $\sa E$, and $\bo c<\bo a<\bo b$ in $\sa F$.
Use (5) to show that $\bo a\ind[M]_{\bo c} \ \bo b$ but $\bo b\nind[M]_{\bo c} \ \bo a.$
\end{ex}

\section{Full Existence and Extension}  \label{s-alg-t^R}

By Proposition \ref{p-alg-indep}, $\ind[a]$ in continuous model theory satisfies symmetry and all axioms for a strict countable independence relation except for base monotonicity and extension.

\begin{rmk}
\label{r-stable-extension}  If $T$ is stable, then the relation $\ind[a]$ in the theory $T^R$ satisfies full existence and extension.
\end{rmk}

\begin{proof} by Theorem 5.1.4 in [BK],  $T^R$ is stable, so it has a unique strict independence relation.  This relation satisfies full existence and is stronger than $\ind[a]$.
Then by Remark \ref{r-weaker}, $\ind[a]$ satisfies full existence.  By Fact \ref{f-fe-vs-ext},
$\ind[a]$ in $T^R$ satisfies extension.
\end{proof}

Our main result in this section is another sufficient condition for algebraic independence in $T^R$ to satisfy full existence and extension

\begin{thm} \label{t-fe} Suppose $T$ has $\acl=\dcl$.  Then the relation $\ind[a]$ in $T^R$ satisfies full existence and extension.
\end{thm}

\begin{proof}
By Fact \ref{f-fe-vs-ext} and Proposition \ref{p-alg-indep}, if $\ind[a]$ over $\cu N$ has full existence, then it has extension.
By Remark \ref{r-acl=dcl}, to prove full existence
we must show that for all small $A,B,C$, there is $A'\equiv_{C} A$ such that
$$[\dcl(A'C)\cap\dcl(BC)=\dcl(C)] \wedge A'\ind[a\BB]_C \ B.$$
In view of Fact \ref{f-definableclosure}  and Remark \ref{r-dcl-B}, we may assume without loss of generality that
$C=\acl(C)$, $A=\acl(AC)\setminus \acl(C)$, and $B=\acl(BC)\setminus \acl(C)$.
Then $C=\dcl(C)$, $A=\dcl(AC)\setminus \dcl(C)$, and $B=\dcl(BC)\setminus \dcl(C)$.    By Fact \ref{f-event-aB}, the relation
$\ind[a\BB] \ \ $ over $\cu N$ has full existence.  Therefore we may also assume that $ A\ind[a\BB]_C \  B.$ By Remark \ref{r-acl=dcl},
$$\dcl_\BB(AC)\cap\dcl_\BB(BC)=\dcl_\BB(C).$$
So it suffices to show that there is $A'\equiv_CA$ such that
$$A'\cap B=\emptyset \wedge \dcl_\BB(A'C)=\dcl_\BB(AC).$$

For each element $\bo a\in A$, we
define $\varepsilon(\bo a)$ as the infimum of all the values $1-\mu(\l a\in\dcl^{\cu M}(D)\rr)$ over all countable $D\subseteq C$.
Note that $\varepsilon(\bo a)=0$ if and only if $\bo a$ is pointwise definable over some countable subset of $C$.
Add a constant symbol for each $\bo a\in A, \bo b\in B$, and $\bo c\in C$.  For each $\bo a\in A$, add a variable $\bo a'$.
Consider the set $\Gamma$ of all conditions of the form
$$\l\theta(\vec{\bo a},\vec{\bo c})\rr=\l\theta(\vec{\bo a}',\vec{\bo c})\rr\wedge
\bigwedge_{i\le|\vec{\bo a}|}d_\BK({\bo a}'_i,\bo b)\ge \varepsilon({\bo a}_i)$$
where $\theta$ is an $L$-formula, $\vec{\bo a}\in A^{<\BN}, \vec {\bo c}\in C^{<\BN}$, and $\bo b\in B$.

\

\emph{Claim 1}.  For every finite subset $\Gamma_0$ of $\Gamma$, there is a set  $A'=\{{\bo a}'\colon \bo a\in A\}$ that satisfies $\Gamma_0$
in $\cu N_{ABC}$.

\

\emph{Proof of Claim 1}:  Let $A_0, B_0, C_0$ be the set of elements of $A, B, C$ respectively that occur in $\Gamma_0$.
Then $A_0, B_0, C_0$ are finite.  If $A_0$ is empty, then $\Gamma_0$ is trivially satisfiable in ${\cu N}_{ABC}$,
so we may assume that $A_0$ is non-empty.   Let
$$A_0=\{{\bo a}_0,\ldots,{\bo a}_n\},\vec{\bo a}=\<{\bo a}_0,\ldots,{\bo a}_n\>,
C_0=\{{\bo c}_0,\ldots,{\bo c}_k\},\vec{\bo c}=\<{\bo c}_0,\ldots,{\bo c}_k\>.$$
Let $\Theta_0$ be the set of all sentences that occur on the left side of an equation in $\Gamma_0$.  Then $\Theta_0$ is finite.  By combining tuples,
 we may assume that each sentence in $\Theta_0$ has the form $\theta(\vec{\bo a},\vec{\bo c})$.

Since the algebraic independence relation over $\cu M$ satisfies full existence, and $T$ has $\acl=\dcl$, for each $\omega\in\Omega$ there exists
$$G_0(\omega)=\{g_0(\omega),\ldots,g_n(\omega)\}\subseteq M$$
such that
$$\tp^{\cu M}(G_0(\omega)/C_0(\omega))=\tp^{\cu M}(A_0(\omega)/C_0(\omega))$$
and
$$ G_0(\omega) \cap B_0(\omega)\subseteq\dcl^{\cu M}(C_0(\omega)).$$
Let $i\le n$.  Whenever $a_i(\omega)\notin\dcl^{\cu M}(C_0(\omega))$, we have $g_i(\omega)\notin\dcl^{\cu M}(C_0(\omega))$, and hence
$g_i(\omega)\notin B_0(\omega)$.

Let $Z=\{\bo 0, \bo 1\}$ be as in Lemma \ref{l-char-function}.
For each $i\le n$ let
$$\sa E_i=\l  a_i\in\dcl^{\cu M}(C_0)\rr.$$
By Fact \ref{f-glue}, for each $i$ there exists a unique element $\bo 1_{\sa E_i}\in\cu K$ that agrees with $\bo 1$ on $\sa E_i$
and agrees with $\bo 0$ on $\neg\sa E_i$.  By applying Condition (5) in Fact \ref{f-neat} to the formula
$$ \bigwedge_{\theta\in\Theta_0}(\theta(\vec u,\vec{\bo c})\leftrightarrow\theta(\vec{\bo a},\vec{\bo c}))\wedge
\bigwedge_{i=0}^n \bigwedge_{\bo b\in B_0} (1_{{\sa E}_i}=\ti 0\rightarrow u_i\ne\bo b),$$
we see that there exists a set
$$G_0=\{{\bo g}_0,\ldots,{\bo g}_n\}\subseteq \cu K$$
such that for each $\omega\in \Omega$, $\theta(\vec{\bo a},\vec{\bo c})\in\Theta_0$, $i\le n$, and $\bo b\in B_0$:
\begin{itemize}
\item
$\cu M\models \theta(\vec g(\omega),\vec c(\omega))\leftrightarrow\theta(\vec a(\omega),\vec c(\omega));$
\item if $a_i(\omega)\notin\dcl^{\cu M}(C_0(\omega))$, then $g_i(\omega)\ne b(\omega)$.
\end{itemize}
It follows that $\l\theta(\vec{\bo g},\vec {\bo c})\rr=\l\theta(\vec{\bo a},\vec {\bo c})\rr$ for each
$\theta(\vec{\bo a},\vec {\bo c})\in\Theta_0$, and that $d_\BK({\bo g}_i,{\bo b})\ge\varepsilon({\bo a}_i)$ for each $i\le n$ and $\bo b\in B_0$.
Therefore $\Gamma_0$ is satisfied by $G_0$ in $\cu N_{ABC}$, and Claim 1 is proved.

\

By saturation, $\Gamma$ is satisfied in $\cu N_{ABC}$ by some set $A'$.  $\Gamma$ guarantees that $A'\equiv_C A$ and  $\dcl_\BB(A'C)=\dcl_\BB(AC).$
It remains to show that for each $\bo a\in A$, $\bo a'\notin B$.  Let $\bo a\in A$.  By hypothesis
we have $\bo a\notin\dcl(C)$.
By Fact \ref{f-dcl3}, either $\bo a$ is not pointwise definable over a countable subset of $C$ and thus $\varepsilon(\bo a)>0$,
or there is a formula $\theta(u,\vec v)$ and a tuple $\vec{\bo c}\in C^{<\BN}$ such that
$$\l\theta(\bo a,\vec{\bo c})\rr\in\fo_\BB(\{\bo a\}\cup C)\setminus\dcl_\BB(C).$$
$\Gamma$ guarantees that $d_\BK(\bo a',B)\ge\varepsilon(\bo a)$, so in the case that $\varepsilon(\bo a)>0$ we have $\bo a'\notin B$.
$\Gamma$ also guarantees that
$$\l\theta(\bo a',\vec{\bo c})\rr=\l\theta(\bo a,\vec{\bo c})\rr,$$
so in the case that $\varepsilon(\bo a)=0$, we have
$$\l\theta(\bo a',\vec{\bo c})\rr=\l\theta(\bo a,\vec{\bo c})\rr\in\dcl_\BB(AC)\setminus\dcl_\BB(C).$$
But we are assuming that
$$\dcl_\BB(AC)\cap\dcl_\BB(BC)=\dcl_\BB(C),$$
so
$$\l\theta(\bo a',\vec{\bo c})\rr\notin\dcl_\BB(BC),$$
and hence $\bo a'\notin B$.  This completes the proof.
\end{proof}

\section{Small Local Character}

In this section we show that if $T$ has $\acl=\dcl$, then  algebraic independence in  $T^R$ has small local character.
In order to do this, we need the pointwise algebraic independence relation $\ind[a\omega] \ \ ,$  which is of interest in its own right
and will be studied further in the next section.

In the following, $\forall^c D$ means ``for all countable $D$'', and $\exists^c D$ means ``there exists a countable $D$''.


\begin{df}  Let $I$ be a ternary relation over $\cu M$ that has monotonicity.  The ternary relation $\ind[I\omega] \ \ $ over $\cu N$ (called \emph{pointwise $I$-independence}) is defined as follows.  For all small $A,B,C$, $ A\ind[I\omega]_C \ B$ if and only if
$$ (\forall^c A'\subseteq A)(\forall^c B'\subseteq B)(\forall^c C'\subseteq C)(\exists^c D\in[C',C])A'\ind[I\omega]_{D} \ B'.$$
 \end{df}

\begin{fact}  \label{f-pointwise-a}  (Consequence of Lemma 4.1.4 in [AGK2].)
If $I$ be a ternary relation over $\cu M$ that has monotonicity, then
 for all countable $A,B,C$,
 $$ A\ind[I \omega]_C \  B\Leftrightarrow \mu(\l A\ind[I]_C B\rr)=1.$$
\end{fact}

 We recall a definition from [AGK2].

\begin{df}  In $T$,  $A\ind[c]_C \ B$ (read ``$C$ covers $A$ in $B$''), is the relation that holds
if and only if for every first order formula
$\varphi(\bar x,\bar y,\bar z)\in[L]$ and all tuples $\bar{ a}\in A^{|\bar x|}$, $\bar{ b}\in B^{|\bar y|}$ and  $\bar{ c}\in C^{|\bar z|}$,
there exists $\bar{ d}\in C^{|\bar y|}$ such that
$$ \cu M\models\varphi(\bar{ a},\bar{ b},\bar{ c})\Rightarrow\varphi(\bar{ a},\bar{ d},\bar{ c}).$$
\end{df}

\begin{fact}  \label{f-covering-omega-small}  (Lemma 7.2.4 in [AGK2].)
In $T^R$, the relation $\ind[c\omega] \ \ $ has small local character.
\end{fact}

\begin{lemma}  \label{l-ind[c]-implies-ind[a]}
In $T$, $\ind[c]\Rightarrow\ind[a].$
\end{lemma}

\begin{proof}  Suppose $A, B, C$ are small  and $A\ind[c]_C B$ in $\cu M$.  Let $e\in\acl^{\cu M}(AC)\cap\acl^{\cu M}(BC)$.  Then there are algebraical formulas
$\varphi(u,\vec x,\vec z), \psi(u,\vec y,\vec w)$ and tuples $\vec a\in A^{<\BN}, \vec b\in B^{<\BN}, \vec c,\vec {c'}\in C^{<\BN}$ such that
$$\cu M\models\varphi(e,\vec a,\vec c)\wedge\psi(e,\vec b,\vec {c'})$$
and
$$(\forall u\in M)[\cu M\models \varphi(u,\vec a,\vec c)\Rightarrow \tp(u/AC)=\tp(e/AC)].$$
Then
$$\cu M\models(\exists u)[\varphi(u,\vec a,\vec c)\wedge\psi(u,\vec b,\vec {c'})].$$
Since $A\ind[c]_C B$, there exists $\vec d\in C^{<\BN}$ such that
$$\cu M\models(\exists u)[\varphi(u,\vec a,\vec c)\wedge\psi(u,\vec d,\vec {c'})].$$
Therefore
$$\cu M\models\psi(e,\vec d,\vec {c'}),$$
so $e\in\acl^{\cu M}(C).$
\end{proof}

\begin{prop}   \label{p-a-omega-small}
In $T^R$, $\ind[a\omega] \ \ $ has small local character.
\end{prop}

\begin{proof}  By Lemma \ref{l-ind[c]-implies-ind[a]}, for all countable $A,B,C\subseteq\cu K$, we have
$$\l A\ind[c]_C B\rr\sqsubseteq\l A\ind[a]_C B\rr.$$
It follows easily that $\ind[c\omega] \ \ \Rightarrow \ind[a\omega] \ \ .$  $\ind[c\omega] \ \ $ has small local character by Fact \ref{f-covering-omega-small},
so by Remark \ref{r-weaker}, $\ind[a\omega] \ \ $ has small local character.
\end{proof}

\begin{prop}  \label{p-I-small}
In $T^R$,  $\ind[a\omega] \ \wedge \ind[a\BB] \ $ has small local character.
\end{prop}

\begin{proof}  By Fact \ref{f-event-dB},  $\ind[d\BB] \ \Rightarrow\ind[a\BB] \ \ ,$ so $\ind[a\omega] \ \wedge \ind[d\BB] \ \Rightarrow\ind[a\omega] \ \wedge \ind[a\BB] \ .$
Then by Remark \ref{r-weaker}, it suffices to show that $\ind[a\omega] \ \wedge \ind[d\BB] \ $ has small local character.

Let $A, B, C_0$ be small subsets of $\cu K$ such that $C_0\subseteq B$ and $|C_0|\le|A|+\aleph_0$.  By Fact \ref{f-event-dB}, $\ind[d\BB] \ $ has small local character, so
there is a set $C_1\in[C_0,B]$ such that $|C_1|\le|A|+\aleph_0$ and $A\ind[d\BB]_{C_1} \ B.$  By Proposition \ref{p-a-omega-small},
there is a set $C_2\in[C_1,B]$ such that $|C_2|\le|A|+\aleph_0$ and $A\ind[a\omega]_{C_2} \ B.$  By Fact \ref{f-event-dB}, $\ind[d\BB] \ $ has base monotonicity,
so $A\ind[d\BB]_{C_2} \ B.$  Therefore $\ind[a\omega] \ \wedge \ind[d\BB] \ $ has small local character.
\end{proof}

\begin{prop}  \label{p-I-implies-a-omega}  The following are equivalent:
\begin{itemize}
\item[(i)]  $T$ has $\acl=\dcl$.
\item[(ii)] In $T^R$, $\ind[a\omega] \ \wedge \ind[a\BB] \ \Rightarrow \ind[a].$
\end{itemize}
\end{prop}

\begin{proof}
Suppose that (i) fails.  Then in $\cu M$ there is a finite set $C$ and an element $a\in\acl^{\cu M}(C)\setminus\dcl^{\cu M}(C)$.
By Fact \ref{f-neat} and saturation, there is  an element $\bo b$ and a finite set $D$ in $\cu K$ such that for each first order formula $\varphi(u,\vec v)$,
if $\cu M\models \varphi(a,C)$ then $\mu(\l\varphi(\bo b,D)\rr)=1$ in $\cu N$.  Therefore in $\cu N$ we have
$\bo b\ind[a\omega]_{D} \ \bo b\wedge \bo b\ind[a\BB]_{D} \ \bo b,$ but $\bo b\notin \acl(D).$   Then $\bo b\nind[a]_D \bo b$, so (ii) fails.

Now suppose (i) holds, and assume that $A\ind[a\omega]_C \ B\wedge A\ind[b\BB]_C \ B.$  We prove that $A\ind[a]_C B.$  By Remark \ref{r-acl=dcl},
it suffices to show that $\dcl(AC)\cap\dcl(BC)\subseteq\dcl(C).$  Let $\bo d\in\dcl(AC)\cap\dcl(BC).$  By Fact \ref{f-dcl3},
\begin{equation}  \label{eq a}
\bo d\in\dcl^\omega(AC),\quad \bo d\in\dcl^\omega(BC),
\end{equation}
and
$$\fo_{\BB}(\bo d AC)\subseteq \dcl_\BB(AC),\quad \fo_{\BB}(\bo d BC)\subseteq \dcl_\BB(BC).$$
By Fact \ref{f-separable},
$$\dcl_{\BB}(\bo d AC)\subseteq \dcl_\BB(AC),\quad \dcl_{\BB}(\bo d BC)\subseteq \dcl_\BB(BC).$$
Then
$$\dcl_{\BB}(\bo d C)\subseteq \dcl_\BB(AC)\cap \dcl_\BB(BC).$$
Since $A\ind[a\BB]_C \ B,$ we have
\begin{equation} \label{eq-b}
\dcl_{\BB}(\bo d C)\subseteq \dcl_\BB(C).
\end{equation}
We next show that
\begin{equation} \label{eq-c}
\bo d\in\dcl^\omega(C).
\end{equation}
By Fact \ref{f-dcl3}, it will then follow that $\bo d\in\dcl(C)$, as required.

By (\ref{eq a}),
there are countable sets $A_0\subseteq A, B_0\subseteq B, C_0\subseteq C$ such that
$$\mu(\l \bo d\in\dcl^{\cu M}(A_0 C_0)\rr)=\mu(\l\bo d\in\dcl^{\cu M}(B_0 C_0)\rr)=1.$$
Since $A\ind[a\omega]_C \ B$, there is a countable set $C_1\in[C_0,C]$ such that
$$\mu(\l A_0\ind[a]_{C_1} B_0\rr)=1.$$
Then
$$\mu(\l \dcl^{\cu M}(A_0 C_0)\cap\dcl^{\cu M}(B_0 C_0)\subseteq\acl^{\cu M}(C_1)\rr)=1,$$
so $\mu(\l\bo d\in\acl^{\cu M}(C_1)\rr)=1$, and hence $\bo d\in\acl^\omega(C)$.  By (i), $\acl^\omega(C)=\dcl^\omega(C)$, so (\ref{eq-c}) holds.
\end{proof}

\begin{thm}  \label{t-small-local-char}
Suppose $T$ has $\acl=\dcl$.  Then the relation $\ind[a]$ in $T^R$ has small local character.
\end{thm}

\begin{proof}  By Proposition \ref{p-I-small}, $\ind[a\omega] \ \wedge \ind[a\BB] \ \ $ has small local character.
By Remark \ref{r-weaker}, Proposition \ref{p-I-implies-a-omega}, and the hypothesis that $T$ has $\acl=\dcl$,
it follows that $\ind[a]$ in $T^R$ has small local character.
 \ \ \end{proof}

Here is a summary of our results about algebraic independence in $T^R$:
For any $T$, algebraic independence in $T^R$ does not satisfy finite character and does not satisfy base monotonicity.
If $T$ has $\acl=\dcl$, then algebraic independence in $T^R$ satisfies all the axioms
for a strict countable independence relation except base monotonicity, and also satisfies finite character and small local character.

\section{Pointwise Algebraic Independence}

In the preceding sections we obtained results about the algebraic independence relation $\ind[a]$ in $T^R$ under the assumption that the underlying first
order theory $T$ has $\acl=\dcl$.  In the general case where $T$ is not assumed to have $\acl=\dcl$,
the pointwise algebraic independence relation $\ind[a\omega] \ \ $ may be an attractive alternative to the algebraic independence
relation $\ind[a]$ in $T^R$.  In this section we investigate the  properties of $\ind[a\omega] \ \ $ in $T^R$ when the underlying
first order theory $T$ is an arbitrary complete theory with models of cardinality $>1$.  We first recall some results from [AGK2].

\begin{fact}  \label{f-pointwise-axioms}  (Special case of Proposition 7.1.4 in [AGK2].)
In $T^R$, $\ind[a\omega] \ \ $ satisfies symmetry and all the axioms for a countable independence relation except perhaps base monotonicity and extension.
Also, if $\ind[a]$ in $T$ has base monotonicity, then so does $\ind[a\omega] \ $ in $T^R$.
\end{fact}

\begin{fact}  \label{f-pointwise-small}  (Corollary 7.2.5 in [AGK2].) In $T^R$, $\ind[a\omega] \ \ $ has small local character.
\end{fact}

\begin{df} A ternary relation $\ind[I]$ has the \emph{countable union property} if whenever
$A, B, C$ are countable, $C=\bigcup_n C_n$, and $C_n\subseteq C_{n+1}$ and $A\ind[I]_{C_n} B$
for each $n$, we have $A\ind[I]_C B$.
\end{df}

\begin{fact}  \label{f-omega-union}  (Special case of Proposition 7.1.6 in [AGK2].)
If the relation $\ind[a]$ in $T$ has monotonicity, finite character, and the countable union property, then the relation $\ind[a\omega] \ \ $ in $T^R$
has finite character.
\end{fact}

\begin{thm}  \label{t-finite-char}
In $T^R$, the relation $\ind[a\omega] \ \ $ has finite character.
\end{thm}

\begin{proof}  It is well-known that $\ind[a]$ in $T$ has monotonicity and finite character.  We show that $\ind[a]$ in $T$ has the countable union property.
Suppose $A, B, C$ are countable, $C=\bigcup_n C_n$, and $C_n\subseteq C_{n+1}$ and
$A\ind[I]_{C_n} B$ for each $n$.  Let $d\in\acl^{\cu M}(AC)\cap\acl^{\cu M}(BC)$.  Then for some $n$ we have $d\in\acl^{\cu M}(AC_n)\cap\acl^{\cu M}(BC_n)$.
 Since $A\ind[a]_{C_n} B$, $d\in\acl^{\cu M}(C_n)$, so $d\in\acl^{\cu M}(C)$.  Therefore $A\ind[a]_C B$, and hence $\ind[a]$ has the countable union property.
So by Fact \ref{f-omega-union}, $\ind[a\omega] \ \ $ has finite character.
\end{proof}

 $\l A\ind[I]_C B\rr\in\cu F$ for all countable $A,B,C\subseteq\cu K$.

\begin{lemma}  \label{l-a-measurable}  For all countable sets $A, B, C\subseteq\cu K$, the set $\l A\ind[a]_C B\rr$ belongs to $\cu F$, and thus is
measurable in the underlying probability space $(\Omega,\cu F,\mu)$.
\end{lemma}

\begin{proof}  Let $\{\varphi_i(u,\vec x)\mid i\in\BN\}$, $\{\psi_j(u,\vec y)\mid j\in\BN\}$, and $\{\chi_k(u,\vec z)\mid k\in\BN\}$ enumerate
all algebraical formulas over the indicated variables.  Then the set $\l A\ind[a]_C B\rr$ is equal to
$$ \bigcap_{i\in\BN}\bigcap_{\vec a\subseteq AC} \bigcap_{j\in\BN}\bigcap_{\vec b\subseteq BC}\bigcup_{k\in\BN}\bigcup_{\vec c\subseteq C}
\l \forall u[\varphi_i(u,\vec {\bo a})\wedge\psi_j(u,\vec{\bo b})\Rightarrow \chi_k(u,\vec{\bo c})]\rr.$$
\end{proof}

\begin{thm}  \label{t-aa-omega-existence}
The  relation $\ind[a\omega] \ \ $ over $\cu N$ satisfies extension and  full existence for all countable sets $A, B, \widehat{B}, C$..
\end{thm}

\begin{proof}
We first prove  full existence for countable sets.  Let $A, B, C$ be countable subsets of $\cu K$.  By Fact \ref{f-event-aB}, the relation
$\ind[a\BB] \ \ $ over $\cu N$ has full existence.  Therefore we may assume that $ A\ind[a\BB]_C \  B.$  By Fact \ref{f-acl=dcl},
$$\dcl_\BB(AC)\cap\dcl_\BB(BC)=\dcl_\BB(C).$$
Since $\ind[a]$ has full existence in $\cu M$,
for each $\omega\in\Omega$ there exists a set $A'_0\subseteq M$ such that $A'_0\equiv_{C(\omega)} A(\omega)$ and $A'_0\ind[a]_{C(\omega)} \ B(\omega)$
in $\cu M$.

Let $\varphi_i(u,A,C)$, $\psi_i(u,B,C)$, and $\chi_i(u,C)$ be enumerations of all algebraical formulas
over the indicated sets (with repetitions) such that for each pair of algebraical formulas $\varphi(u,A,C)$ and $\psi(u,B,C)$ there
exists an $i$ such that $(\varphi_i,\psi_i)=(\varphi,\psi)$.  Whenever $\omega\in\Omega$, $A'_0\subseteq\cu M$, and $A'_0\ind[a]_{C(\omega)} \ B(\omega)$ in $\cu M$,
for each $i\in \BN$ there exists $j\in\BN$ such that
\begin{equation} \label{eq-ind}
 \cu M \models \forall u [\varphi_i(u,A'_0,C(\omega))\wedge\psi_i(u,B(\omega),C(\omega))\rightarrow\chi_j(u,C(\omega))].
 \end{equation}
Let $\BN^0=\{\emptyset\}$ and $\sa E_\emptyset=\Omega$.
For each $n>0$ and $n$-tuple $s=\langle s(0),\ldots,s(n-1)\rangle$ in $\BN^n$, let $\sa E_s$ be the set of all $\omega\in\Omega$ such that
for some set $A'_0\subseteq\cu M$, $A'_0\equiv_{C(\omega)} B(\omega)$ and (\ref{eq-ind}) holds whenever $i<n$ and $j=s(i)$.

Let $L'$ be the signature formed by adding to $L$ the constant symbols
$$\{ k_a, k_b, k_c \ : \ a\in A, b\in B, c\in C\}.$$
For each $\omega\in\Omega$, $(\cu M,A(\omega),B(\omega),C(\omega))$ will be the $L'$-structure where $k_a, k_b, k_c$ are interpreted by
$a(\omega), b(\omega), c(\omega)$.  Form $L''$ by adding to $L'$ countably many additional constant symbols $\{k'_a \, : \, a\in A\}$ that will be used for elements of
a countable subset $A'_0$ of $\cu M$.

Then for each $n>0$ and $s\in \BN^n$, there is a countable set of sentences
$\Gamma_s$ of $L''$ such that for each $\omega$, $\omega\in \sa E_s$ if and only if $\Gamma_s$ is satisfiable in $(\cu M,A(\omega),B(\omega),C(\omega))$.
Since $\cu M$ is $\aleph_1$-saturated, $\Gamma_s$ is satisfiable if and only if it is finitely satisfiable in $(\cu M,A(\omega),B(\omega),C(\omega))$.
It follows that the set $\sa E_s$ belongs to the $\sigma$-algebra $\cu F$.  Moreover, since $\ind[a]$ has full existence in $\cu M$, for each $n$ and $s\in \BN^n$ we have
$$\Omega\doteq\bigcup\{\sa E_t \colon t\in \BN^n\},\quad  \sa E_s \doteq\bigcup\{ \sa E_{sk}\colon k\in\BN\},$$
where $sk$ is the $(n+1)$-tuple formed by adding $k$ to the end of $s$.  We now cut down the sets $\sa E_s$ to sets $\sa F_s\in\cu F$ such that:
\begin{itemize}
\item[(a)] $\sa F_\emptyset=\Omega$;
\item[(b)] $\sa F_s\subseteq\sa E_s$ whenever $s\in \BN^n$;
\item[(c)] $\sa F_s\cap\sa F_t=\emptyset$ whenever $s,t\in \BN^n$ and $s\ne t$;
\item[(d)] $\sa F_s\doteq\bigcup\{\sa F_{sk}\colon k\in\BN\}$ whenever $s\in\BN^n$.
\end{itemize}
This can be done as follows.  Assuming $\sa F_s$ has been defined for each $s\in\BN^n$. we let
$$ \sa F_{sk}=\sa F_s\cap(\sa E_{sk}\setminus\bigcup_{j<k} \sa F_{sj}).$$
Now let $\theta_i(A,C)$ enumerate all first order sentences with constants for the elements of $AC$.  Let $\Sigma$ and $\Delta$ be the following countable sets of sentences of $(L'')^R$:
$$\Sigma=\{\l \theta_i(A',C)\rr\doteq\l\theta_i(A,C)\rr\colon i\in\BN\}.$$
$$\Delta=\{\sa F_s\sqsubseteq \l\forall u[\varphi_i(u,A',C))\wedge\psi_i(u, B, C))\rightarrow \chi_{s(i)}(u,C))]\rr\colon s\in\BN^{<\BN}, i<|s|\}.$$
It follows from Fact \ref{f-neat} (5) and conditions (a)--(d) above that $\Sigma\cup\Delta$ is finitely satisfiable in $\cu N_{ABC}$.  Then by saturation, there is a set $A'$
that satisfies $\Sigma\cup\Delta$ in $\cu N_{ABC}$.  Since $A'$ satisfies $\Sigma$, we have $A'\equiv_C A$.
The sentences $\Delta$ guarantee that $A'\ind[a\omega]_C \ B$.

By the proof of Fact \ref{f-fe-vs-ext} (1) (see the Appendix of [Ad1]), invariance, monotonicity, transitivity, normality,  symmetry,
and full existence for all countable sets implies extension for all countable sets.  Then by the preceding paragraphs and Fact \ref{f-pointwise-axioms},
$\ind[a\omega] \ \ $ satisfies  extension for all countable sets.
\end{proof}

\begin{question}  Does $\ind[a\omega] \ \ $ satisfy extension for countable $A,B,C$ and small $\widehat{B}$?
\end{question}

\begin{question}  Does $\ind[a\omega] \ \ $ satisfy full existence and/or extension?
\end{question}

We conclude by showing that  the relations $\ind[a]$ and $\ind[a\omega] \ \ $ are incomparable except in trivial cases.

\begin{prop}  \label{p-pointwise-algebraic-vs-algebraic}

\noindent\begin{enumerate}
\item[(i)] $\ind[a\omega] \ \ $ is not anti-reflexive.
\item[(ii)] $\ind[a\omega] \ \Rightarrow \ind[a]$ always fails in $\cu N$.
\item[(iii)] $\ind[a]\Rightarrow\ind[a\omega] \ \ $ holds in $\cu N$ if and only if the models of $T$ are finite.
\end{enumerate}
\end{prop}

\begin{proof}  (i) and (ii) Let $Z=\{\bo 0, \bo 1\}$ be as in Lemma \ref{l-char-function}.  Let $\sa D$ be a set in $\cu F$
of measure $\mu(\sa D)=1/2$, and let $\bo 1_{\sa D}$ agree with $\bo 1$ on $\sa D$ and agree with $\bo 0$ on $\neg \sa D$.  
Then $\bo 1_{\sa D}\ind[a\omega]_Z \ \bo 1_{\sa D}$, but $\bo 1_{\sa D}\notin\acl(Z).$  Therefore $\bo 1_{\sa D}\nind[a]_Z \bo 1_{\sa D}$
and $\ind[a]$ is not anti-reflexive.

(iii)  If $\cu M$ is finite, then $\acl^{\cu M}(\emptyset)=M$, so $A\ind[a]_C B$ always holds in $\cu M$.  Therefore $A\ind[a\omega]_C \ B \ $ always holds in $\cu N$,
and hence $\ind[a]\Rightarrow\ind[a\omega] \ \ $ holds in $\cu N$.

For the other direction, assume $\cu M$ is infinite.  By saturation of $\cu M$, there exist elements $0,1,a,b\in M$ such that
$$ 0\ne 1,\quad a\notin\acl^{\cu M}(01), \quad\tp(a/\acl^{\cu M}(01))=\tp(b/\acl^{\cu M}(01)), \quad a\ind[a]_{01} \, b.$$
By a routine transfinite induction using Fact \ref{f-neat} and the saturation of $\cu N$, there is a mapping $a\mapsto\ti a$ from $M$ into $\cu K$ such that
for each tuple $a_0,a_1,\ldots$ in $M$ and formula $\varphi(v_0,v_1,\ldots)$ of $L$, if $\cu M\models \varphi(a_0,a_1,\ldots)$ then
$\mu(\l\varphi(\ti a_0,\ti a_1,\ldots)\rr)=1$ in $\cu N$.  Let $\ti M=\{\ti a\mid a\in M\}.$

To simplify notation, suppose first that $T$ already has a constant symbol for each element of $\acl(01)$.  Then
$\acl^{\cu M}(01)=\acl^{\cu M}(\emptyset)$, so
$$ 0\ne 1, \quad a\notin\acl^{\cu M}(\emptyset),\quad\tp(a)=\tp(b), \quad a\ind[a]_{\emptyset} \, b \quad\mbox{ in } \cu M,$$
$$ \mu(\l {\ti 0}\ne {\ti 1}\rr)=1,\quad {\ti a}\notin\dcl(\emptyset),\quad\tp({\ti a})=\tp({\ti b}), \quad {\ti a}\ind[a]_{\emptyset} \, {\ti b} \quad\mbox{ in } \cu N.$$
By Results \ref{f-acl=dcl} and \ref{f-separable}, for each $A\subseteq \ti M$,
$$\acl(A)=\dcl(A)=\cl(\fo(A))=\fo(A)\subseteq \ti M.$$

Let $\sa E\in\cu E$ be an event of measure $\mu(\sa E)=1/2$.
Let $ \bo c$ agree with $\ti 1$ on $\sa E$ and $\ti 0$ on $\neg\sa E.$ Let $\bo d$ agree with $\ti a$ on $\neg \sa E$ and with $\ti b$ on $\sa E$
(see the figure).
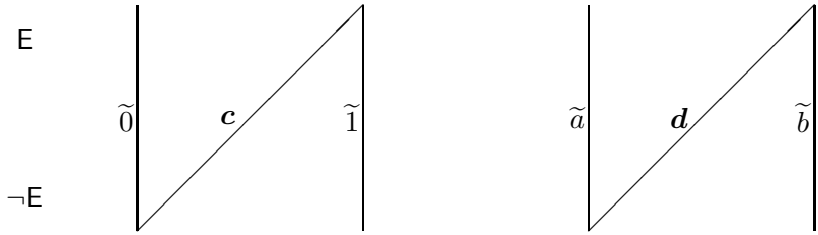
\begin{center}
\setlength{\unitlength}{1.5mm}
\begin{picture}(80,20)
\put(10,0){\line(0,1){20}}
\put(30,0){\line(0,1){20}}
\put(50,0){\line(0,1){20}}
\put(70,0){\line(0,1){20}}
\put(10,0){\line(1,1){20}}
\put(50,0){\line(1,1){20}}

\put(0,3){\makebox(0,0){$\neg\sa E$}}
\put(0,17){\makebox(0,0){$\sa E$}}
\put(9,10){\makebox(0,0){$\ti 0$}}
\put(29,10){\makebox(0,0){$\ti 1$}}
\put(49,10){\makebox(0,0){$\ti a$}}
\put(69,10){\makebox(0,0){$\ti b$}}
\put(18,10){\makebox(0,0){$\bo c$}}
\put(58,10){\makebox(0,0){$\bo d$}}
\end{picture}
\end{center}

\emph{Claim 1}: $\ti a\ind[a]_{\emptyset} \bo{cd} $ in $\cu N$.

\emph{Proof of Claim 1}:  Suppose $ x\in\acl({\ti a})\cap\acl(\bo{cd})$ in $\cu N$.  Then $x\in \dcl(\ti a)$,
so $x=\ti z$ for some $z\in\dcl^{\cu M}(a)$.
Moreover, $x\in\dcl(\bo c\bo d)$, so $x\in\dcl^\omega(\bo c\bo d)$, and hence $x(\omega)\in\dcl^{\cu M}(1b)=\dcl^{\cu M}(b)$ for all $\omega\in\sa E$.
Therefore $z\in\dcl^{\cu M}(b)$.  Since ${\ti a}\ind[a]_{\emptyset} \, {\ti b}$ in $\cu N$, we have
$x\in\acl(\ti a)\cap\acl(\ti b)=\acl(\emptyset)$.
\medskip

\emph{Claim 2}: $\ti a\nind[a\omega]_{\emptyset} \ \bo{cd} $ in $\cu N$.

\emph{Proof of Claim 2}:
For all $\omega\in\neg\sa E$ we have $\bo d(\omega)=\ti a(\omega)$, so
$$\ti a(\omega)\in\acl^{\cu M}(\ti a(\omega))\cap \acl^{\cu M}(\bo{cd}(\omega))\setminus \acl^{\cu M}(\emptyset),$$
and hence $\omega\notin\l \ti a\ind[a]_\emptyset \bo c\bo d\rr$.
Therefore $\mu(\l \ti a\ind[a]_\emptyset\bo c\bo d\rr)\le 1/2$, so $\ti a\nind[a\omega]_{\emptyset} \ \bo{cd} $.

By Claims 1 and 2, $\ind[a]\Rightarrow\ind[a\omega] \ \ $ fails in $\cu N$.

We now turn to the general case where $T$ need not have a constant symbol for each element of $\acl(01)$.
Our argument above shows that $\ti a\ind[a]_{\ti 0 \ti 1} \bo{cd}$ but $\ti a\nind[a\omega]_{\ti 0 \ti 1} \ \bo{cd}$ in $\cu N$, so
$\ind[a]\Rightarrow\ind[a\omega] \ \ $ still fails in $\cu N$.
\end{proof}

\section*{References}


\vspace{2mm}

[Ad1]  Hans Adler. Explanation of Independence.  PH. D. Thesis, Freiburg, AxXiv:0511616 (2005).

[Ad2]  Hans Adler.  A Geometric Introduction to Forking and Thorn-forking.  J. Math. Logic 9 (2009), 1-21.

[AGK1]  Uri Andrews, Isaac Goldbring, and H. Jerome Keisler.  Definable Closure in Randomizations.
To appear, Annals of Pure and Applied Logic.  Available online at www.math.wisc.edu/$\sim$Keisler.

[AGK2]  Uri Andrews, Isaac Goldbring, and H. Jerome Keisler.   Randomizing o-minimal Theories.  Submitted.  Available online at www.math.wisc.edu/$\sim$Keisler.

[Be]  Ita\"i{} Ben Yaacov.  On Theories of Random Variables.  Israel J. Math 194 (2013), 957-1012.

[BBHU]  Ita\"i{} Ben Yaacov, Alexander Berenstein,
C. Ward Henson and Alexander Usvyatsov. Model Theory for Metric Structures.  In Model Theory with Applications to Algebra and Analysis, vol. 2,
London Math. Society Lecture Note Series, vol. 350 (2008), 315-427.

[BK] Ita\"i{} Ben Yaacov and H. Jerome Keisler. Randomizations of Models as Metric Structures.
Confluentes Mathematici 1 (2009), pp. 197-223.

[BU] Ita\"i{} Ben Yaacov and Alexander Usvyatsov. Continuous first order logic and local stability. Transactions of the American
Mathematical Society 362 (2010), no. 10, 5213-5259.

[CK]  C.C.Chang and H. Jerome Keisler.  Model Theory.  Dover 2012.

[EG]  Clifton Ealy and Isaac Goldbring.  Thorn-Forking in Continuous Logic.  Journal of Symbolic Logic 77 (2012), 63-93.

%
%




[GL]  Rami Grossberg and Olivier Lessman.  Dependence Relation in Pregeometries. Algebra Universalis 44 (2000), 199-216.


[Ke] H. Jerome Keisler.  Randomizing a Model.  Advances in Math 143 (1999), 124-158.




\end{document}